\documentclass[10pt]{article}

\usepackage{amsthm, amsmath, amsfonts, amssymb, latexsym}
\usepackage[utf8]{inputenc}
\usepackage[english]{babel}
\usepackage[hidelinks]{hyperref}
\usepackage{orcidlink}
\usepackage{graphicx}
\usepackage{pict2e}

\usepackage[backend=biber,style=trad-unsrt,doi=true,url=false,isbn=false]{biblatex}
\theoremstyle{plain}
\newtheorem{theorem}{Theorem}
\newtheorem{lemma}{Lemma}
\newtheorem{corollary}{Corollary}

\newtheorem{openproblem}{Open problem}
\newtheorem{conjecture}{Conjecture}

\theoremstyle{definition}

\theoremstyle{remark}

\begin{document}

    \title{A counterexample to the Albertson-Berman conjecture about induced forests in planar graphs}
    \author{Mikhail Makarov\orcidlink{0009-0004-3604-4547}\footnote{mikhail.makarov.math@gmail.com}}
    \date{}
    \maketitle

    \begin{abstract}
        For a graph $G$, denote by $a(G)$ the number of vertices in the largest induced forest in $G$. The Albertson-Berman conjecture, which had been open since 1979, states that $a(G) \geq \frac{n}{2}$ for every simple planar graph $G$ on $n$ vertices. Although the Albertson-Berman conjecture was recently resolved in the negative by constructing a counterexample with the help of AI, we independently found a counterexample to the Albertson-Berman conjecture without AI and present it in this article. Our counterexample is on $39$ vertices with $a(G)=19$. Finally, we indicate, without a full proof, a variant of this construction with a larger number of vertices, but with a slightly smaller ratio $\frac{a(G)}{n}=\frac{37}{76}$.
    \end{abstract}

    \section{Disclaimer}
        For the record, I found the counterexample to the AB conjecture described in this article on August 2, 2026 without AI. I was in the process of writing the draft of this article and of trying to further optimize the ratio $\frac{a(G)}{n}$ when I discovered that another preprint~\cite{jung_ab_conjecture_counterexample} was published on August 11, 2026 also claiming to have found a counterexample to the AB conjecture with a slightly smaller ratio $\frac{a(G)}{n}$ than mine, but with the help of AI. I decided to quickly finalize my draft and to publish it immediately as a preprint in its present somewhat unfinished form: with the full proof of the counterexample, but without finishing the research of further optimizing the ratio $\frac{a(G)}{n}$ that I planned, without optimizing several arguments (in particular, the cumbersome case analysis in several lemmas), with little editing, and without polishing the text. Even though it seems to be too late to claim priority for resolving the AB conjecture.

        As evidence that I had the counterexample before the publication of~\cite{jung_ab_conjecture_counterexample}, I include in this preprint a screenshot of the graph $G_{14}$ from my construction with a timestamp of August 3.

    \section{Introduction}
        For a graph $G$, denote by $a(G)$ the number of vertices in the largest induced forest in $G$.
        
        \begin{conjecture}[Albertson-Berman, AB, \cite{ab_conjecture}]\label{conjecture:ab}
            For every simple planar graph $G$ on $n$ vertices, we have $a(G) \geq \frac{n}{2}$.
        \end{conjecture}

        The AB conjecture had been open since 1979.

        Denote by $\alpha(G)$ the independence number of $G$, which is the number of vertices in the largest independent set in $G$. Also, denote by $\alpha_k(G)$ the largest possible combined number of vertices in $k$ disjoint independent sets in $G$, or, equivalently, the number of vertices in the largest $k$-colorable induced subgraph in $G$. In particular, we have $\alpha_1(G)=\alpha(G)$, and $\alpha_2(G)$ is the number of vertices in the largest bipartite induced subgraph in $G$.

        For $k \in \{1,2,3,4\}$, the Four Color Theorem~\cite{appel_haken_4_color_theorem_part1, appel_haken_4_color_theorem_part2} implies, by taking the $k$ largest color classes in a proper $4$-coloring of $G$, the lower bound $\alpha_k(G) \geq \frac{kn}{4}$ for every planar graph $G$ on $n$ vertices. In particular, we have the lower bounds $\alpha(G) \geq \frac{n}{4}$ and $\alpha_2(G) \geq \frac{n}{2}$.

        This is the only known proof of the lower bound $\alpha(G) \geq \frac{n}{4}$. It is an open problem to find a proof of the lower bound $\alpha(G) \geq \frac{n}{4}$ that does not use the Four Color Theorem. If the AB conjecture were true, it would have implied the lower bounds $\alpha(G) \geq \frac{n}{4}$ and $\alpha_2(G) \geq \frac{n}{2}$ as well because a forest is bipartite and hence an induced forest can be split into two independent sets, which immediately would have given $\alpha_2(G) \geq \frac{n}{2}$, and, taking the largest independent set of the two, would have given $\alpha(G) \geq \frac{n}{4}$. So, assuming that the proof of the AB conjecture itself had not used the Four Color Theorem, this would have given such an alternative proof of the lower bound $\alpha(G) \geq \frac{n}{4}$ without using the Four Color Theorem.

        The AB conjecture also can be viewed as a strengthening of the lower bound $\alpha_2(G) \geq \frac{n}{2}$ because an induced forest is effectively two independent sets with an additional condition that their union does not contain cycles.

        This additional condition was also previously considered for color classes (which are independent sets) in proper colorings of graphs, and a result on such colorings gives the best known lower bound on $a(G)$ for planar graphs. Specifically, a vertex coloring of a graph is called \emph{acyclic} if it is proper and the union of any two color classes induces a forest, or, equivalently, if it is proper and any cycle contains at least $3$ colors. The result that every planar graph has an acyclic $5$-coloring~\cite{borodin_acyclic_5-coloring} implies, by taking the two largest color classes, the lower bound $a(G) \geq \frac{2}{5}n$, which is currently the best known.

        The known tight example for the AB conjecture is a disjoint union of an arbitrary number of copies of $K_4$, which attains the exact equality $a(G)=\frac{n}{2}$.

        So, the best possible lower bound on $a(G)$ for planar graphs must be somewhere between $\frac{2}{5}n$ and $\frac{n}{2}$. The AB conjecture states that it is exactly $\frac{n}{2}$.

        Although the AB conjecture was recently resolved in the negative by constructing a counterexample with the help of AI~\cite{jung_ab_conjecture_counterexample}, we independently found a counterexample to the AB conjecture without AI and present it in this article.

    \section{The augmentation of a restricting graph}
        We call a graph \emph{$(1,2)$-edge-restricting} with a \emph{restricting edge} and two \emph{restricting vertices} incident to that edge if every largest induced forest in that graph contains either exactly $1$ or exactly $2$ restricting vertices (and does not contain the other restricting vertices). If, in addition, for each subset of $1$ or $2$ vertices of the restricting vertices, there exists a largest induced forest in that graph containing the vertices of that subset and not containing any other restricting vertices, then we call such a $(1,2)$-edge-restricting graph \emph{all-realizable}.

        For a graph $G$ on $n$ vertices, denote $r_a(G)=\frac{a(G)}{n}$.

        \begin{lemma}\label{lemma:augmenting_1-2-edge-restricting_to_counterexample}
            Let $H$ be a $(1,2)$-edge-restricting planar graph on $n \geq 2$ vertices. We construct a new graph $G$ by taking a triangle and attaching to each of its three edges a copy of $H$ by the restricting edge (identifying the restricting edge in each copy with an edge of the triangle). Then $G$ is a planar graph on $3n-3$ vertices with $a(G) \leq 3a(H)-2$ and $r_a(G) \leq r_a(H)-\frac{1}{n-1}\left(\frac{2}{3}-r_a(H)\right)$. If, in addition, $H$ is all-realizable, then we have the exact equalities $a(G) = 3a(H)-2$ and $r_a(G) = r_a(H)-\frac{1}{n-1}\left(\frac{2}{3}-r_a(H)\right)$.
        \end{lemma}
        \begin{figure}[htb]
            \centering
            \setlength{\unitlength}{0.8mm}
            \begin{picture}(69.268145,64.016539)(-15.884072,-28.125)
                \Line(0,0)(18.75,32.47595)
                \Line(18.75,32.47595)(37.5,0)
                \Line(37.5,0)(0,0)

                \qbezier(0,0)(18.75,-56.25)(37.5,0)
                \qbezier(37.5,0)(76.838925,44.362975)(18.75,32.47595)
                \qbezier(18.75,32.47595)(-39.338925,44.362975)(0,0)

                \put(0,0){\circle*{1.5}}
                \put(18.75,32.47595){\circle*{1.5}}
                \put(37.5,0){\circle*{1.5}}

                \put(-4,-2){$u$}
                \put(17.5,34){$v$}
                \put(39,-2){$w$}

                \put(-0.3678,21.8630){\makebox(0,0){$H_{uv}$}}
                \put(37.8678,21.8630){\makebox(0,0){$H_{vw}$}}
                \put(18.75,-11.25){\makebox(0,0){$H_{wu}$}}
            \end{picture}
            \caption{The augmentation of a $(1,2)$-edge-restricting graph.}
            \label{figure:augmentation}
        \end{figure}
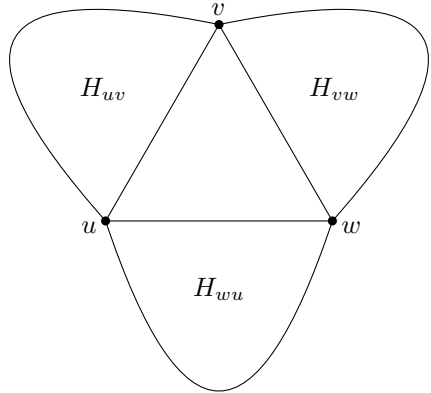
        \begin{proof}
            Denote the vertices of the triangle to which the copies of $H$ are attached by $u$, $v$, $w$. Denote by $H_{uv}$, $H_{vw}$, $H_{wu}$ the copies of $H$ attached to the edges $uv$, $vw$, $wu$, respectively.
            
            The planarity of $G$ follows from the fact that, for each of the three copies of $H$, we can take an embedding of $H$ into the plane with the restricting edge incident to the external face, and attach the copies of it to the triangle $uvw$ in three disjoint regions of the exterior of $uvw$.
            
            Consider an arbitrary induced forest $F$ in $G$. Each of the copies of $H$ has at most $a(H)$ vertices in $F$. We consider the cases of how many vertices among $u$, $v$, $w$ belong to $F$.

            Case 1. None of the vertices $u$, $v$, $w$ belongs to $F$. Then, for each of the three attached copies of $H$, $F$ does not contain any of the restricting vertices in that copy, which implies that the restriction of $F$ to that copy is not the largest induced forest in that copy. This means that the number of vertices of $F$ in each copy of $H$ is at most $a(H)-1$. Therefore, in total, $F$ has at most $3(a(H)-1) = 3a(H)-3 < 3a(H)-2$ vertices.

            Case 2. Exactly one of the vertices $u$, $v$, $w$ belongs to $F$. Without loss of generality, we can assume that $u$ belongs to $F$ and $v$ and $w$ do not belong to $F$. Then $F$ contains $1$ restricting vertex of $H_{uv}$, none of the restricting vertices of $H_{vw}$, and $1$ restricting vertex of $H_{wu}$. Therefore, $H_{uv}$ has at most $a(H)$ vertices of $F$, $H_{vw}$ has at most $a(H)-1$ vertices of $F$, and $H_{wu}$ has at most $a(H)$ vertices of $F$. Also, the vertex $u$ belonging to $F$ is common between $H_{uv}$ and $H_{wu}$, and none of the other vertices of $F$ is common between the copies of $H$. So, when summing the number of vertices of $F$ in the copies of $H$, we need to subtract $1$ for the common vertex. Therefore, in total, $F$ has at most $a(H)+(a(H)-1)+a(H)-1=3a(H)-2$ vertices.

            Case 3. Exactly two of the vertices $u$, $v$, $w$ belong to $F$. Without loss of generality, we can assume that $u$ and $v$ belong to $F$ and $w$ does not belong to $F$. Then each of the three copies of $H$ has at most $a(H)$ vertices of $F$, the vertex $u$ belonging to $F$ is common between $H_{uv}$ and $H_{wu}$, the vertex $v$ belonging to $F$ is common between $H_{uv}$ and $H_{vw}$, and none of the other vertices of $F$ is common between the copies of $H$. So, when summing the number of vertices of $F$ in the copies of $H$, we need to subtract $2$ for the two common vertices. Therefore, in total, $F$ has at most $3a(H)-2$ vertices.

            The case where all three vertices $u$, $v$, $w$ belong to $F$ is impossible because they form a cycle.

            So, in all possible cases, $F$ has at most $3a(H)-2$ vertices. Therefore, we have $a(G) \leq 3a(H)-2$, as claimed. Then we have $r_a(G) = \frac{a(G)}{3n-3} \leq \frac{3a(H)-2}{3n-3} = \frac{3r_a(H)n-2}{3n-3} = \frac{1}{n-1}\left(r_a(H)n-\frac{2}{3}\right) = \frac{1}{n-1}\left(r_a(H)(n-1)+r_a(H)-\frac{2}{3}\right) = r_a(H)-\frac{1}{n-1}\left(\frac{2}{3}-r_a(H)\right)$, as claimed.

            Now, assume that $H$ is all-realizable. Then we can choose largest induced forests $F_{uv}$, $F_{vw}$, $F_{wu}$ of size $a(H)$ in $H_{uv}$, $H_{vw}$, $H_{wu}$, respectively, such that $F_{uv}$ contains both $u$ and $v$, $F_{vw}$ contains $v$ and does not contain $w$, and $F_{wu}$ contains $u$ and does not contain $w$. Let us prove that we can merge $F_{uv}$, $F_{vw}$, $F_{wu}$ without creating a cycle, that is, their union, which we denote by $F$, is an induced forest in $G$. Suppose, to the contrary, that $F$ contains a cycle $C$. This cycle $C$ cannot be contained inside only one of the copies of $H$ and hence it must pass through at least two copies. If $C$ contains a vertex in $H_{wu}$, then the only way for it to get out of $H_{wu}$ is through the vertex $u$ because $w$ does not belong to $F$. But then $C$ can never return back to $H_{wu}$ because $w$ does not belong to $F$ and $u$ was already visited on the way out of $H_{wu}$. Therefore, $C$ cannot contain vertices in $H_{wu}$. Similarly, we can prove the symmetrical case that $C$ cannot contain vertices in $H_{vw}$. Therefore, $C$ can contain vertices from $H_{uv}$, which is a contradiction with the previously established fact that it must pass through at least two copies of $H$. This concludes the proof that $F$ is an induced forest.
            
            We count that $F$ contains exactly $3a(H)-2$ vertices (subtracting $2$ for the two common vertices). Therefore, we have $a(G) \geq |V(F)|=3a(H)-2$. Combining this with the opposite inequality $a(G) \leq 3a(H)-2$ obtained above, we get the exact equality $a(G)=3a(H)-2$, as claimed. The inequality for $r_a(G)$ also becomes the exact equality because the only inequality used in its derivation was $a(G) \leq 3a(H)-2$, which, as we have just proved, became the exact equality. So, we have $r_a(G) = r_a(H)-\frac{1}{n-1}\left(\frac{2}{3}-r_a(H)\right)$, as claimed.
        \end{proof}

        Lemma~\ref{lemma:augmenting_1-2-edge-restricting_to_counterexample} means that, by augmenting a $(1,2)$-edge-restricting planar graph $H$ into $G$, we strictly decrease the ratio $r_a$ if $r_a(H) < \frac{2}{3}$. In particular, if we find a $(1,2)$-edge-restricting planar graph $H$ with $r_a(H)=\frac{1}{2}$, then, augmenting it, we would obtain a counterexample to the AB conjecture.

    \section{The restricting graph}
        We consider the planar graph on $8$ vertices shown in Figure~\ref{figure:G_8}, denote it by $G_8$, and denote its vertices as shown in that figure. We call the vertices $c_1$ and $c_2$ \emph{central} and the remaining vertices $u_1$, $u_2$, $u_3$, $u_4$, $u_5$, $u_6$ \emph{boundary}.
        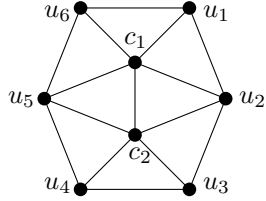
\begin{figure}[htb]
            \centering
            \setlength{\unitlength}{1.2mm}
            \begin{picture}(33.5,21.5)(13.25,12.25)
                \Line(24,33)(36,33)
                \Line(24,33)(30,27)
                \Line(24,33)(20,23)
                \Line(30,27)(20,23)
                \Line(30,27)(36,33)
                \Line(30,27)(40,23)
                \Line(30,27)(30,19)
                \Line(20,23)(30,19)
                \Line(20,23)(24,13)
                \Line(24,13)(30,19)
                \Line(24,13)(36,13)
                \Line(30,19)(40,23)
                \Line(30,19)(36,13)
                \Line(36,13)(40,23)
                \Line(40,23)(36,33)
                \put(24,33){\circle*{1.5}}
                \put(30,27){\circle*{1.5}}
                \put(20,23){\circle*{1.5}}
                \put(24,13){\circle*{1.5}}
                \put(30,19){\circle*{1.5}}
                \put(36,13){\circle*{1.5}}
                \put(40,23){\circle*{1.5}}
                \put(36,33){\circle*{1.5}}
                \put(28.8,29){$c_1$}
                \put(29.3,16.3){$c_2$}
                \put(37.5,32){$u_1$}
                \put(41.5,22.5){$u_2$}
                \put(37.5,13){$u_3$}
                \put(20,13){$u_4$}
                \put(16,22.5){$u_5$}
                \put(20,32){$u_6$}
            \end{picture}
            \caption{The graph $G_8$.}
            \label{figure:G_8}
        \end{figure}

        \begin{lemma}\label{lemma:G_8_a}
            We have $a(G_8)=5$.
        \end{lemma}
        \begin{proof}
            We directly verify that the $5$ vertices $u_1$, $u_2$, $u_3$, $u_4$, $u_5$ induce a forest in $G_8$.

            It remains to prove that there are no induced forests on at least $6$ vertices. Suppose, to the contrary, that there exists an induced forest $F$ on at least $6$ vertices. We observe that $F$ cannot contain all the $6$ boundary vertices because they form a cycle. Therefore, $F$ contains at most $5$ boundary vertices and hence at least $1$ central vertex. Since $G_8$ is symmetric under the simultaneous swapping between $c_1$ and $c_2$, $u_1$ and $u_3$, $u_6$ and $u_4$, we can assume without loss of generality that $F$ contains $c_1$. Then $F$ cannot contain both $u_1$ and $u_2$ because $u_1u_2c_1$ is a triangle. Similarly, $F$ cannot contain both $u_5$ and $u_6$ because $u_5u_6c_1$ is a triangle. So, $F$ contains at most $2$ vertices among $u_1$, $u_2$, $u_5$, $u_6$. This implies that $F$ must contain all the remaining $4$ vertices $c_1$, $c_2$, $u_3$, $u_4$ in order to have at least $6$ vertices in total. But then $c_2u_3u_4$ is a triangle in $F$, which is a contradiction that concludes the proof.
        \end{proof}

        \begin{lemma}\label{lemma:G_8_induced_forests_5_1}
            Every induced forest on $5$ vertices in $G_8$ contains at least $2$ vertices among $u_1$, $u_2$, $u_3$ that belong to the same connected component of that forest and at least $2$ vertices among $u_4$, $u_5$, $u_6$ that belong to the same connected component of that forest.
        \end{lemma}
        \begin{proof}
            It is enough to prove the part for $u_1$, $u_2$, $u_3$, as the part for $u_4$, $u_5$, $u_6$ is symmetrical and can be proved similarly.
            
            Suppose, to the contrary, that there exists an induced forest $F$ on $5$ vertices in $G_8$ without at least $2$ vertices among $u_1$, $u_2$, $u_3$ belonging to the same connected component of $F$.

            Consider the case where $F$ contains a central vertex. Since $G_8$ is symmetric under the simultaneous swapping between $c_1$ and $c_2$, $u_1$ and $u_3$, $u_6$ and $u_4$, we can assume without loss of generality that $F$ contains $c_1$. Then $F$ cannot contain both $u_5$ and $u_6$ because then $u_5u_6c_1$ would be a triangle in $F$. Consider the subcase where $F$ contains $u_2$. Then $F$ cannot contain $u_1$ because $u_1$ and $u_2$ are adjacent and hence they would be the two required vertices among $u_1$, $u_2$, $u_3$. Similarly, $F$ cannot contain $u_3$ because then $u_2$ and $u_3$ would be the two required vertices. So, we found $3$ vertices that $F$ cannot contain: $u_1$, $u_3$, and either $u_5$ or $u_6$. Therefore, in order to have $5$ vertices in total, $F$ must contain all the remaining $5$ vertices. In particular, it must contain $c_2$. But then $u_2c_1c_2$ is a triangle in $F$. This concludes the proof for the subcase where $F$ contains $u_2$. Therefore, $F$ cannot contain $u_2$. Now, consider the subcase where $F$ does not contain $u_1$. Then we have $3$ vertices that $F$ does not contain: $u_2$, $u_1$, and either $u_5$ or $u_6$. Therefore, in order to have $5$ vertices in total, $F$ must contain all the remaining $5$ vertices. In particular, it must contain the vertices $c_2$, $u_3$, $u_4$ that form a triangle. This concludes the proof for the subcase where $F$ does not contain $u_1$. Therefore, $F$ must contain $u_1$. Then $F$ cannot contain $u_6$ because then $u_1u_6c_1$ would be a triangle in $F$. Now, $F$ cannot contain both $c_2$ and $u_5$ because then $c_1c_2u_5$ would be a triangle in $F$. So, we found $3$ vertices that $F$ does not contain: $u_2$, $u_6$, and either $c_2$ or $u_5$. Therefore, in order to have $5$ vertices in total, $F$ must contain all the remaining $5$ vertices. In particular, it must contain the vertices $u_3$ and $u_4$. Then $F$ cannot contain $c_2$ because then $u_3u_4c_2$ would be a triangle in $F$. Then it must contain the remaining vertex $u_5$. Then $u_1$ and $u_3$ are the required two vertices among $u_1$, $u_2$, $u_3$ because they are connected by the path $u_1c_1u_5u_4u_3$ in $F$. This concludes the proof for the case where $F$ contains a central vertex.

            The only remaining case is where all vertices of $F$ are boundary. Since $F$ contains $5$ vertices and there are exactly $6$ boundary vertices, $F$ omits exactly $1$ boundary vertex. If $F$ contains both $u_1$ and $u_2$, then $u_1$ and $u_2$ are the two required vertices among $u_1$, $u_2$, $u_3$. Therefore, $F$ must omit either $u_1$ or $u_2$. Since, $F$ can omit only $1$ boundary vertex, it must contain all other $4$ boundary vertices $u_3$, $u_4$, $u_5$, $u_6$. If $F$ omits $u_1$, then $u_2$ and $u_3$ are the two required vertices. If $F$ omits $u_2$, then $u_1$ and $u_3$ are the two required vertices connected by the path $u_3u_4u_5u_6u_1$ in $F$. This concludes the proof of the case where all vertices of $F$ are boundary and the proof of the entire lemma.
        \end{proof}

        \begin{lemma}\label{lemma:G_8_induced_forests_5_2}
            Every induced forest on $5$ vertices in $G_8$ contains either both $u_1$ and $u_6$ or both $u_3$ and $u_4$.
        \end{lemma}
        \begin{proof}
            Suppose, to the contrary, that there exists an induced forest $F$ on $5$ vertices in $G_8$ that contains at most $1$ vertex among $u_1$ and $u_6$ and at most $1$ vertex among $u_3$ and $u_4$. Then, in order to have $5$ vertices in total, $F$ must omit at most $1$ vertex of the remaining $4$ vertices $c_1$, $c_2$, $u_2$, $u_5$. Observe that $F$ cannot contain all $4$ of these vertices or omit $u_5$ because then $c_1c_2u_2$ would be a triangle in $F$. Also, $F$ cannot omit $u_2$ because then $c_1c_2u_5$ would be a triangle in $F$. Therefore, $F$ can only omit either $c_1$ or $c_2$. These two cases are symmetrical under the simultaneous swapping between $c_1$ and $c_2$, $u_1$ and $u_3$, $u_6$ and $u_4$. So, we can assume, without loss of generality that $F$ omits $c_2$. Then, $F$ cannot contain $u_1$ because then $u_1u_2c_1$ would be a triangle in $F$. Similarly, $F$ cannot contain $u_6$ because then $u_5u_6c_1$ would be a triangle in $F$. So, $F$ contains the vertices $u_2$, $c_1$, $u_5$, omits $c_2$, $u_1$, $u_6$, and contains at most $1$ vertex among $u_3$ and $u_4$. So, $F$ contains at most $4$ vertices in total, which is a contradiction that concludes the proof.
        \end{proof}
        
        \begin{lemma}\label{lemma:G_8_induced_forests_4}
            For every induced forest $F$ on at least $4$ vertices in $G_8$, at least one of the following options hold:
            \begin{enumerate}
                \item $F$ contains at least $2$ vertices among $u_1$, $u_6$, $u_3$, $u_4$ that belong to the same connected component of $F$;

                \item $F$ contains the vertices $u_1$, $u_2$, $u_4$, $u_5$;

                \item $F$ contains the vertices $u_2$, $u_3$, $u_5$, $u_6$;

                \item $F$ contains the vertices $u_2$, $u_5$, and at least one vertex among $u_1$, $u_6$, $u_3$, $u_4$ such that all three of them belong to the same connected component of $F$;

                \item $F$ contains the vertices $u_2$, $c_1$, $u_6$, $u_4$;

                \item $F$ contains the vertices $u_1$, $c_1$, $u_5$, $u_3$;

                \item $F$ contains the vertices $u_3$, $c_2$, $u_5$, $u_1$;

                \item $F$ contains the vertices $u_4$, $c_2$, $u_2$, $u_6$.
            \end{enumerate}
        \end{lemma}
        \begin{proof}
            Suppose, to the contrary, that there exists an induced forest $F$ on at least $4$ vertices in $G_8$ not satisfying any of the $8$ options in the statement of the lemma. Since $F$ does not satisfy option 1, it contains at most $1$ vertex among $u_1$ and $u_6$ and at most $1$ vertex among $u_3$ and $u_4$.

            Consider the case where $F$ contains both central vertices $c_1$ and $c_2$. Then $F$ cannot contain $u_2$ because then $c_1c_2u_2$ would be a triangle in $F$. Similarly, $F$ cannot contain $u_5$ because then $c_1c_2u_5$ would be a triangle in $F$. So, $F$ contains $c_1$, $c_2$, at most $1$ vertex among $u_1$ and $u_6$, at most $1$ vertex among $u_3$ and $u_4$ and omits $u_2$ and $u_5$. Therefore, in order for $F$ to contain at least $4$ vertices in total, it must contain exactly $1$ vertex among $u_1$ and $u_6$, which we denote by $x$, and exactly $1$ vertex among $u_3$ and $u_4$, which we denote by $y$. Then option 1 from the statement of the lemma holds for the vertices $x$ and $y$ that are connected through $c_1$ and $c_2$.
            
            Consider the case where $F$ contains exactly $1$ central vertex. Since $G_8$ is symmetric under the simultaneous swapping between $c_1$ and $c_2$, $u_1$ and $u_3$, $u_6$ and $u_4$, we can assume without loss of generality that $F$ contains $c_1$ and does not contain $c_2$. In order for $F$ to contain at least $4$ vertices in total, it must contain either $u_2$ or $u_5$. Because of the symmetry, we can assume without loss of generality that $F$ contains $u_2$. Then $F$ cannot contain $u_1$ because then $u_2u_1c_1$ would be a triangle in $F$. Consider the subcase where $F$ contains $u_6$. Then $F$ cannot contain $u_5$ because then $u_5u_6c_1$ would be a triangle in $F$. So, $F$ contains $u_2$, $c_1$, $u_6$, at most $1$ vertex among $u_3$ and $u_4$, and omits $u_1$, $u_5$, $c_2$. In order for $F$ to contain at least $4$ vertices in total, it must contain exactly $1$ vertex among $u_3$ and $u_4$. If $F$ contains $u_3$, then option 1 from the statement of the lemma holds for the vertices $u_3$ and $u_6$ connected through $u_2$ and $c_1$. If $F$ contains $u_4$, then option 5 holds. This concludes the proof of the subcase where $F$ contains $u_6$. Therefore, $F$ does not contain $u_6$. In order for $F$ to contain at least $4$ vertices in total, it must contain $u_5$ and exactly $1$ vertex among $u_3$ and $u_4$. If $F$ contains $u_4$, then option 4 holds for the vertices $u_2$, $u_5$ and $u_4$. If $F$ contains $u_3$, then option 4 holds for the vertices $u_2$, $u_5$, $u_3$. This concludes the proof of the case where $F$ contains exactly $1$ central vertex.

            The only remaining case is where all vertices of $F$ are boundary. Since we previously established that $F$ contains at most $1$ vertex among $u_1$ and $u_6$ and at most $1$ vertex among $u_3$ and $u_4$, in order for $F$ to contain at least $4$ vertices in total, it must contain both of the remaining two boundary vertices $u_2$ and $u_5$, and it must contain exactly $1$ vertex among $u_1$ and $u_6$ and exactly $1$ vertex among $u_3$ and $u_4$. Now, we enumerate all possible pairs of vertices one of which is among $u_1$ and $u_6$ and another is among $u_3$ and $u_4$. If $F$ contains $u_1$ and $u_4$, then option 2 from the statement of the lemma holds. If $F$ contains $u_6$ and $u_3$, then option 3 holds. If $F$ contains $u_1$ and $u_3$, then option 1 holds for the vertices $u_1$ and $u_3$ that are connected through $u_2$. If $F$ contains $u_6$ and $u_4$, then option 1 holds for the vertices $u_6$ and $u_4$ that are connected through $u_5$. So, for all possible pairs, one of the options from the statement of the lemma holds. This concludes the proof for the case where all vertices of $F$ are boundary and the proof of the entire lemma.
        \end{proof}

        We add $4$ vertices $v_1$, $v_2$, $v_3$, $v_4$ to $G_8$ as shown in Figure~\ref{figure:G_12} and denote the resulting planar graph on $12$ vertices by $G_{12}$.
        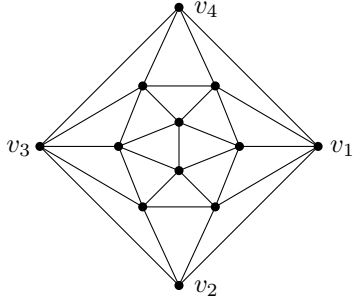
\begin{figure}[htb]
            \centering
            \setlength{\unitlength}{0.8mm}
            \begin{picture}(57.5,47.5)(1.25,-0.75)
                \Line(53,23)(30,0)
                \Line(53,23)(36,13)
                \Line(53,23)(40,23)
                \Line(53,23)(36,33)
                \Line(53,23)(30,46)
                \Line(24,33)(7,23)
                \Line(24,33)(30,46)
                \Line(24,33)(36,33)
                \Line(24,33)(30,27)
                \Line(24,33)(20,23)
                \Line(30,27)(20,23)
                \Line(30,27)(36,33)
                \Line(30,27)(40,23)
                \Line(30,27)(30,19)
                \Line(20,23)(30,19)
                \Line(20,23)(24,13)
                \Line(20,23)(7,23)
                \Line(7,23)(24,13)
                \Line(7,23)(30,0)
                \Line(7,23)(30,46)
                \Line(24,13)(30,19)
                \Line(24,13)(36,13)
                \Line(24,13)(30,0)
                \Line(30,19)(40,23)
                \Line(30,19)(36,13)
                \Line(36,13)(30,0)
                \Line(36,13)(40,23)
                \Line(40,23)(36,33)
                \Line(36,33)(30,46)
                \put(53,23){\circle*{1.5}}
                \put(24,33){\circle*{1.5}}
                \put(30,27){\circle*{1.5}}
                \put(20,23){\circle*{1.5}}
                \put(7,23){\circle*{1.5}}
                \put(24,13){\circle*{1.5}}
                \put(30,19){\circle*{1.5}}
                \put(36,13){\circle*{1.5}}
                \put(40,23){\circle*{1.5}}
                \put(36,33){\circle*{1.5}}
                \put(30,46){\circle*{1.5}}
                \put(30,0){\circle*{1.5}}
                \put(55,22){$v_1$}
                \put(32.5,-1){$v_2$}
                \put(1.5,22){$v_3$}
                \put(32.5,45){$v_4$}
            \end{picture}
            \caption{The graph $G_{12}$.}
            \label{figure:G_12}
        \end{figure}

        \begin{lemma}\label{lemma:G_12_a}
            We have $a(G_{12})=6$.
        \end{lemma}
        \begin{proof}
            We directly verify that the $6$ vertices $v_2$, $u_3$, $u_2$, $u_5$, $u_6$, $v_4$ induce a forest in $G_{12}$.

            It remains to prove that there are no induced forests in $G_{12}$ on at least $7$ vertices. Suppose, to the contrary, that there exists an induced forest $F$ on at least $7$ vertices. By Lemma~\ref{lemma:G_8_a}, $F$ contains at most $5$ vertices in the internal $G_8$. Therefore, $F$ must contain at least $2$ vertices among $v_1$, $v_2$, $v_3$, $v_4$. On the other hand, $F$ cannot contain all $4$ vertices $v_1$, $v_2$, $v_3$, $v_4$ because they form a cycle. Therefore, $F$ contains either $2$ or $3$ vertices among $v_1$, $v_2$, $v_3$, $v_4$.

            Consider the case where $F$ contains exactly $2$ vertices among $v_1$, $v_2$, $v_3$, $v_4$. Then it contains exactly $5$ vertices in the internal $G_8$. Then, by Lemma~\ref{lemma:G_8_induced_forests_5_1}, $F$ contains at least $2$ vertices, which we denote by $x$ and $y$, among $u_1$, $u_2$, $u_3$ that are connected by a path $p$ in $F$ inside the internal $G_8$. Therefore, $F$ cannot contain $v_1$ because then $v_1xpy$ would be a cycle in $F$. Similarly, by Lemma~\ref{lemma:G_8_induced_forests_5_1}, $F$ contains at least $2$ vertices, which we denote by $z$ and $t$, among $u_4$, $u_5$, $u_6$ that are connected by a path $q$ in $F$ inside the internal $G_8$. Therefore, $F$ cannot contain $v_3$ because then $v_3zqt$ would be a cycle in $F$. Then, the only way for $F$ to contain exactly $2$ vertices among $v_1$, $v_2$, $v_3$, $v_4$ is to contain both $v_2$ and $v_4$. By Lemma~\ref{lemma:G_8_induced_forests_5_2}, $F$ contains either both $u_1$ and $u_6$ or both $u_3$ and $u_4$. If it contains both $u_1$ and $u_6$, then $v_4u_1u_6$ is a triangle in $F$, and, if it contains both $u_3$ and $u_4$, then $v_2u_3u_4$ is a triangle in $F$, which is a contradiction that concludes the proof of the case where $F$ contains exactly $2$ vertices among $v_1$, $v_2$, $v_3$, $v_4$.

            The only remaining case is where $F$ contains exactly $3$ vertices among $v_1$, $v_2$, $v_3$, $v_4$ and thus omits exactly $1$ vertex among them. Then it contains at least $4$ vertices in the internal $G_8$. Because of the symmetry, we can consider only two possibilities: that either $F$ omits $v_3$ or $F$ omits $v_4$. The other two possibilities are considered similarly. By Lemma~\ref{lemma:G_8_induced_forests_4}, we have one of the $8$ options listed there. We consider each of them separately.

            Consider the subcase where option 1 holds, that is, where $F$ contains at least $2$ vertices among $u_1$, $u_6$, $u_3$, $u_4$ that belong to the same connected component of $F$. If one of these two vertices is $u_3$, then $v_1v_2u_3$ is a triangle in $F$. If these two vertices are $u_1$ and $u_6$ and $F$ contains $v_4$, then $v_4u_1u_6$ is a triangle in $F$. If these two vertices are $u_1$ and $u_6$ and $F$ omits $v_4$ and thus contains $v_3$, then $u_1v_1v_2v_3u_6$ is a cycle in $F$. If these two vertices are $u_1$ and $u_4$ connected by a path $p$ in $F$ inside the internal $G_8$, then $v_1u_1pu_4v_2$ is a cycle in $F$. If these two vertices are $u_6$ and $u_4$ and $F$ contains $v_3$, then $v_2v_3u_4$ is a triangle in $F$. If these two vertices are $u_6$ and $u_4$ connected by a path $p$ in $F$ inside the internal $G_8$ and $F$ omits $v_3$ and thus contains $v_4$, then $v_4v_1v_2u_4pu_6$ is a cycle in $F$. This concludes the proof of the subcase where option 1 holds.

            Consider the subcase where option 2 holds, that is, where $F$ contains the vertices $u_1$, $u_2$, $u_4$, $u_5$. Then $v_1u_1u_2$ is a triangle in $F$.

            Consider the subcase where option 3 holds, that is, where $F$ contains the vertices $u_2$, $u_3$, $u_5$, $u_6$. Then $v_1u_2u_3$ is a triangle in $F$.

            Consider the subcase where option 4 holds, that is, where $F$ contains the vertices $u_2$, $u_5$, and at least one vertex among $u_1$, $u_6$, $u_3$, $u_4$ such that all three of them belong to the same connected component of $F$. If the third vertex is $u_3$, then $v_1u_2u_3$ is a triangle in $F$. If the third vertex is $u_4$ and $u_2$ is connected with $u_4$ by a path $p$ in $F$ inside the internal $G_8$, then $v_1u_2pu_4v_2$ is a cycle in $F$. If the third vertex is $u_1$, then $v_1u_2u_1$ is a triangle in $F$. If the third vertex is $u_6$ and $F$ contains $v_3$, then $v_3u_5u_6$ is a triangle in $F$. If the third vertex is $u_6$ and $F$ omits $v_3$ and thus contains $v_4$, and $u_2$ and $u_6$ are connected by a path $p$ in $F$ inside the internal $G_8$, then $v_1u_2pu_6v_4$ is a cycle in $F$. This concludes the proof of the subcase where option 4 holds.

            Consider the subcase where option 5 holds, that is, where $F$ contains the vertices $u_2$, $c_1$, $u_6$, $u_4$. If $F$ contains $v_3$, then $v_2v_3u_4$ is a triangle in $F$. If $F$ omits $v_3$ and thus contains $v_4$, then $v_1u_2c_1u_6v_4$ is a cycle in $F$.

            Consider the subcase where option 6 holds, that is, where $F$ contains the vertices $u_1$, $c_1$, $u_5$, $u_3$. Then $v_1v_2u_3$ is a triangle in $F$.

            Consider the subcase where option 7 holds, that is, where $F$ contains the vertices $u_3$, $c_2$, $u_5$, $u_1$. Then $v_1v_2u_3$ is a triangle in $F$.

            Consider the subcase where option 8 holds, that is, where $F$ contains the vertices $u_4$, $c_2$, $u_2$, $u_6$. Then $v_1v_2u_4c_2u_2$ is a cycle in $F$.

            So, in all subcases, we found a cycle in $F$, which is a contradiction. This concludes the proof of the case where $F$ contains exactly $3$ vertices among $v_1$, $v_2$, $v_3$, $v_4$ and of the entire lemma.
        \end{proof}

        \begin{lemma}\label{lemma:G_12_induced_forests_6}
            Every induced forest in $G_{12}$ on $6$ vertices contains either $v_2$ or $v_4$ or both $v_1$ and $v_3$ such that they belong to the same connected component of that forest.
        \end{lemma}
        \begin{proof}
            Suppose, to the contrary, that there exists an induced forest $F$ on $6$ vertices in $G_{12}$ without the required properties. Then $F$ does not contain $v_2$ and does not contain $v_4$. Also, either $F$ contains at most $1$ vertex among $v_1$ and $v_3$ or $F$ contains both of them, but they belong to different connected components of $F$. We consider these two cases separately.

            Consider the case where $F$ contains at most $1$ vertex among $v_1$ and $v_3$. If $F$ contains neither of them, then all $6$ vertices of $F$ must be inside the internal $G_8$, which contradicts Lemma~\ref{lemma:G_8_a}. Therefore, $F$ must contain exactly $1$ vertex among $v_1$ and $v_3$. Because of the symmetry, without loss of generality, we can assume that $F$ contains $v_1$ and omits $v_3$. Then $F$ contains $5$ vertices in the internal $G_8$. By Lemma~\ref{lemma:G_8_induced_forests_5_1}, $F$ contains at least $2$ vertices, which we denote by $x$ and $y$, among $u_1$, $u_2$, $u_3$ that are connected by a path $p$ in $F$ inside the internal $G_8$. Then $v_1xpy$ is a cycle in $F$, which is a contradiction that concludes the proof of the case where $F$ contains at most $1$ vertex among $v_1$ and $v_3$.

            The only remaining case is where $F$ contains both of $v_1$ and $v_3$, but they belong to different connected components of $F$. By Lemma~\ref{lemma:G_8_induced_forests_4}, we have one of the $8$ options listed there. We consider each of them separately.

            Consider the subcase where option 1 holds, that is, where $F$ contains at least $2$ vertices among $u_1$, $u_6$, $u_3$, $u_4$ that belong to the same connected component of $F$. If these two vertices are $u_1$ and $u_6$, then $v_1$ and $v_3$ are connected through them, contradicting the assumption that $v_1$ and $v_3$ belong to different connected components of $F$. Similarly, if these two vertices are $u_3$ and $u_4$, then $v_1$ and $v_3$ are connected through them. If these two vertices are $u_1$ and $u_3$ and they are connected by a path $p$ in $F$ inside the internal $G_8$, then $v_1u_1pu_3$ is a cycle in $F$. If these two vertices are $u_6$ and $u_4$ and they are connected by a path $p$ in $F$ inside the internal $G_8$, then $v_3u_6pu_4$ is a cycle in $F$. If these two vertices are $u_1$ and $u_4$ and they are connected by a path $p$ in $F$ inside the internal $G_8$, then $v_1$ and $v_3$ are connected by the path $v_1u_1pu_4v_3$ in $F$, contradicting the assumption. If these two vertices are $u_6$ and $u_3$ and they are connected by a path $p$ in $F$ inside the internal $G_8$, then $v_1$ and $v_3$ are connected by the path $v_1u_3pu_6v_3$ in $F$, contradicting the assumption. This concludes the proof of the subcase where option 1 holds.

            Consider the subcase where option 2 holds, that is, where $F$ contains the vertices $u_1$, $u_2$, $u_4$, $u_5$. Then $v_1u_1u_2$ is a triangle in $F$.

            Consider the subcase where option 3 holds, that is, where $F$ contains the vertices $u_2$, $u_3$, $u_5$, $u_6$. Then $v_1u_2u_3$ is a triangle in $F$.

            Consider the subcase where option 4 holds, that is, where $F$ contains the vertices $u_2$, $u_5$, and at least one vertex among $u_1$, $u_6$, $u_3$, $u_4$ such that all three of them belong to the same connected component of $F$. Denote by $p$ the path inside the internal $G_8$ connecting $u_2$ and $u_5$. Then the path $v_1u_2pu_5v_3$ in $F$ connects $v_1$ and $v_3$, contradicting the assumption that they belong to different connected components of $F$.

            Consider the subcase where option 5 holds, that is, where $F$ contains the vertices $u_2$, $c_1$, $u_6$, $u_4$. Then the path $v_1u_2c_1u_6v_3$ in $F$ connects $v_1$ and $v_3$, contradicting the assumption that they belong to different connected components of $F$.

            Consider the subcase where option 6 holds, that is, where $F$ contains the vertices $u_1$, $c_1$, $u_5$, $u_3$. Then the path $v_1u_1c_1u_5v_3$ in $F$ connects $v_1$ and $v_3$, contradicting the assumption that they belong to different connected components of $F$.

            Consider the subcase where option 7 holds, that is, where $F$ contains the vertices $u_3$, $c_2$, $u_5$, $u_1$. Then the path $v_1u_3c_2u_5v_3$ in $F$ connects $v_1$ and $v_3$, contradicting the assumption that they belong to different connected components of $F$.

            Consider the subcase where option 8 holds, that is, where $F$ contains the vertices $u_4$, $c_2$, $u_2$, $u_6$. Then the path $v_1u_2c_2u_4v_3$ in $F$ connects $v_1$ and $v_3$, contradicting the assumption that they belong to different connected components of $F$.

            So, in all subcases, we found either a cycle in $F$ or a path connecting $v_1$ and $v_3$, which is a contradiction. This concludes the proof of the case where $F$ contains both of $v_1$ and $v_3$, but they belong to different connected components of $F$, and of the entire lemma.
        \end{proof}

        We add two vertices $w_1$ and $w_2$ to $G_{12}$ as shown in Figure~\ref{figure:G_14} and denote the resulting planar graph on $14$ vertices by $G_{14}$.
        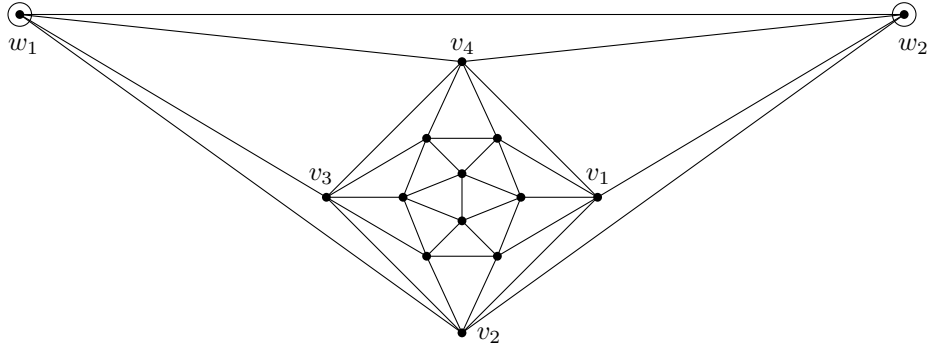
\begin{figure}[htb]
            \centering
            \setlength{\unitlength}{0.78mm}
            \begin{picture}(154,56.75)(-47,-0.75)
                \Line(53,23)(30,0)
                \Line(53,23)(36,13)
                \Line(53,23)(40,23)
                \Line(53,23)(36,33)
                \Line(53,23)(30,46)
                \Line(53,23)(105,54)
                \Line(24,33)(7,23)
                \Line(24,33)(30,46)
                \Line(24,33)(36,33)
                \Line(24,33)(30,27)
                \Line(24,33)(20,23)
                \Line(30,27)(20,23)
                \Line(30,27)(36,33)
                \Line(30,27)(40,23)
                \Line(30,27)(30,19)
                \Line(20,23)(30,19)
                \Line(20,23)(24,13)
                \Line(20,23)(7,23)
                \Line(7,23)(24,13)
                \Line(7,23)(30,0)
                \Line(7,23)(-45,54)
                \Line(7,23)(30,46)
                \Line(24,13)(30,19)
                \Line(24,13)(36,13)
                \Line(24,13)(30,0)
                \Line(30,19)(40,23)
                \Line(30,19)(36,13)
                \Line(36,13)(30,0)
                \Line(36,13)(40,23)
                \Line(40,23)(36,33)
                \Line(36,33)(30,46)
                \Line(30,46)(-45,54)
                \Line(30,46)(105,54)
                \Line(30,0)(105,54)
                \Line(30,0)(-45,54)
                \Line(105,54)(-45,54)
                \put(53,23){\circle*{1.5}}
                \put(24,33){\circle*{1.5}}
                \put(30,27){\circle*{1.5}}
                \put(20,23){\circle*{1.5}}
                \put(7,23){\circle*{1.5}}
                \put(24,13){\circle*{1.5}}
                \put(30,19){\circle*{1.5}}
                \put(36,13){\circle*{1.5}}
                \put(40,23){\circle*{1.5}}
                \put(36,33){\circle*{1.5}}
                \put(30,46){\circle*{1.5}}
                \put(30,0){\circle*{1.5}}
                \put(105,54){\circle*{1.5}}
                \put(105,54){\circle{4}}
                \put(-45,54){\circle*{1.5}}
                \put(-45,54){\circle{4}}
                \put(51,26){$v_1$}
                \put(32.5,-1){$v_2$}
                \put(4,26){$v_3$}
                \put(28,48){$v_4$}
                \put(-47,48){$w_1$}
                \put(104,48){$w_2$}
            \end{picture}
            \caption{The all-realizable $(1,2)$-edge-restricting planar graph $G_{14}$. The restricting vertices are encircled.}
            \label{figure:G_14}
        \end{figure}

        \begin{lemma}\label{lemma:G_14_a}
            We have $a(G_{14})=7$.
        \end{lemma}
        \begin{proof}
            We directly verify that the $7$ vertices $w_1$, $v_2$, $u_3$, $u_2$, $u_5$, $u_6$, $v_4$ induce a forest in $G_{14}$.

            It remains to prove that there are no induced forests in $G_{14}$ on at least $8$ vertices. Suppose, to the contrary, that there exists an induced forest $F$ on at least $8$ vertices. By Lemma~\ref{lemma:G_12_a}, the internal $G_{12}$ contains at most $6$ vertices of $F$. Therefore, $F$ must contain both $w_1$ and $w_2$ and the internal $G_{12}$ must contain exactly $6$ vertices of $F$. By Lemma~\ref{lemma:G_12_induced_forests_6}, $F$ contains either $v_2$ or $v_4$ or both $v_1$ and $v_3$ such that they belong to the same connected component of $F$. If $F$ contains $v_2$, then $w_1w_2v_2$ is a triangle in $F$. If $F$ contains $v_4$, then $w_1w_2v_4$ is a triangle in $F$. If $F$ contains both $v_1$ and $v_3$ connected by a path $p$ in $F$ inside the internal $G_{12}$, then $w_1v_3pv_1w_2$ is a cycle in $F$. So, in all cases, we found a cycle in $F$, which is a contradiction. This concludes the proof of the lemma.
        \end{proof}

        \begin{lemma}\label{lemma:G_14_induced_forests_7}
            The graph $G_{14}$ is all-realizable $(1,2)$-edge-restricting with the restricting vertices $w_1$ and $w_2$.
        \end{lemma}
        \begin{proof}
            By Lemma~\ref{lemma:G_14_a}, we have $a(G_{14})=7$. Therefore, to prove that $G_{14}$ is $(1,2)$-edge-restricting with the restricting vertices $w_1$ and $w_2$, we need to prove that every induced forest in $G_{14}$ on $7$ vertices contains either $w_1$ or $w_2$. Suppose, to the contrary, that there exists an induced forest $F$ in $G_{14}$ on $7$ vertices that omits both $w_1$ and $w_2$. Then $F$ is entirely contained inside the internal $G_{12}$, which contradicts the fact that $a(G_{12})=6$ established in Lemma~\ref{lemma:G_12_a}. Therefore, $G_{14}$ is $(1,2)$-edge-restricting with the restricting vertices $w_1$ and $w_2$.
            
            To prove that $G_{14}$ is all-realizable, we directly verify that the vertices $w_1$, $v_2$, $u_3$, $u_2$, $u_5$, $u_6$, $v_4$ induce a forest in $G_{14}$ on $7$ vertices and contain only $w_1$ among $w_1$ and $w_2$, the vertices $w_2$, $v_2$, $u_3$, $u_2$, $u_5$, $u_6$, $v_4$ induce a forest in $G_{14}$ on $7$ vertices and contain only $w_2$ among $w_1$ and $w_2$, and the vertices $w_1$, $w_2$, $v_1$, $u_3$, $u_4$, $u_5$, $u_6$ induce a forest in $G_{14}$ on $7$ vertices and contain both $w_1$ and $w_2$. This concludes the proof of the fact that $G_{14}$ is all-realizable and of the entire lemma.
        \end{proof}

    \section{The counterexample}
        Now, we augment the graph $G_{14}$ as described in Lemma~\ref{lemma:augmenting_1-2-edge-restricting_to_counterexample}. Specifically, we take a triangle and attach to each of its three edges a copy of $G_{14}$ by the restricting edge (identifying the restricting edge in each copy with an edge of the triangle), and denote the resulting planar graph on $39$ vertices by $G_{39}$. It is shown in Figure~\ref{figure:G_39}.

        \begin{figure}[htb]
            \centering
            \setlength{\unitlength}{0.64mm}
            \begin{picture}(170.0308,185.4038)(-10.0154,-54.75)

                \Line(98,-31)(75,-54)
                \Line(98,-31)(81,-41)
                \Line(98,-31)(85,-31)
                \Line(98,-31)(81,-21)
                \Line(98,-31)(75,-8)
                \Line(98,-31)(150,0)
                \Line(69,-21)(52,-31)
                \Line(69,-21)(75,-8)
                \Line(69,-21)(81,-21)
                \Line(69,-21)(75,-27)
                \Line(69,-21)(65,-31)
                \Line(75,-27)(65,-31)
                \Line(75,-27)(81,-21)
                \Line(75,-27)(85,-31)
                \Line(75,-27)(75,-35)
                \Line(65,-31)(75,-35)
                \Line(65,-31)(69,-41)
                \Line(65,-31)(52,-31)
                \Line(52,-31)(69,-41)
                \Line(52,-31)(75,-54)
                \Line(52,-31)(0,0)
                \Line(52,-31)(75,-8)
                \Line(69,-41)(75,-35)
                \Line(69,-41)(81,-41)
                \Line(69,-41)(75,-54)
                \Line(75,-35)(85,-31)
                \Line(75,-35)(81,-41)
                \Line(81,-41)(75,-54)
                \Line(81,-41)(85,-31)
                \Line(85,-31)(81,-21)
                \Line(81,-21)(75,-8)
                \Line(75,-8)(0,0)
                \Line(75,-8)(150,0)
                \Line(75,-54)(150,0)
                \Line(75,-54)(0,0)
                \Line(150,0)(0,0)
                \put(0,0){\circle*{1.5}}
                \put(150,0){\circle*{1.5}}
                \put(98,-31){\circle*{1.5}}
                \put(69,-21){\circle*{1.5}}
                \put(75,-27){\circle*{1.5}}
                \put(65,-31){\circle*{1.5}}
                \put(52,-31){\circle*{1.5}}
                \put(69,-41){\circle*{1.5}}
                \put(75,-35){\circle*{1.5}}
                \put(81,-41){\circle*{1.5}}
                \put(85,-31){\circle*{1.5}}
                \put(81,-21){\circle*{1.5}}
                \put(75,-8){\circle*{1.5}}
                \put(75,-54){\circle*{1.5}}

                \Line(127.8468,100.3704)(159.2654,91.9519)
                \Line(127.8468,100.3704)(145.0071,90.6480)
                \Line(127.8468,100.3704)(134.3468,89.1121)
                \Line(127.8468,100.3704)(127.6866,80.6480)
                \Line(127.8468,100.3704)(119.4282,68.9519)
                \Line(127.8468,100.3704)(75,129.9038)
                \Line(133.6866,70.2557)(150.8468,60.5333)
                \Line(133.6866,70.2557)(119.4282,68.9519)
                \Line(133.6866,70.2557)(127.6866,80.6480)
                \Line(133.6866,70.2557)(135.8827,78.4519)
                \Line(133.6866,70.2557)(144.3468,71.7916)
                \Line(135.8827,78.4519)(144.3468,71.7916)
                \Line(135.8827,78.4519)(127.6866,80.6480)
                \Line(135.8827,78.4519)(134.3468,89.1121)
                \Line(135.8827,78.4519)(142.8109,82.4519)
                \Line(144.3468,71.7916)(142.8109,82.4519)
                \Line(144.3468,71.7916)(151.0071,80.2557)
                \Line(144.3468,71.7916)(150.8468,60.5333)
                \Line(150.8468,60.5333)(151.0071,80.2557)
                \Line(150.8468,60.5333)(159.2654,91.9519)
                \Line(150.8468,60.5333)(150,0)
                \Line(150.8468,60.5333)(119.4282,68.9519)
                \Line(151.0071,80.2557)(142.8109,82.4519)
                \Line(151.0071,80.2557)(145.0071,90.6480)
                \Line(151.0071,80.2557)(159.2654,91.9519)
                \Line(142.8109,82.4519)(134.3468,89.1121)
                \Line(142.8109,82.4519)(145.0071,90.6480)
                \Line(145.0071,90.6480)(159.2654,91.9519)
                \Line(145.0071,90.6480)(134.3468,89.1121)
                \Line(134.3468,89.1121)(127.6866,80.6480)
                \Line(127.6866,80.6480)(119.4282,68.9519)
                \Line(119.4282,68.9519)(150,0)
                \Line(119.4282,68.9519)(75,129.9038)
                \Line(159.2654,91.9519)(75,129.9038)
                \Line(159.2654,91.9519)(150,0)
                \Line(75,129.9038)(150,0)
                \put(150,0){\circle*{1.5}}
                \put(75,129.9038){\circle*{1.5}}
                \put(127.8468,100.3704){\circle*{1.5}}
                \put(133.6866,70.2557){\circle*{1.5}}
                \put(135.8827,78.4519){\circle*{1.5}}
                \put(144.3468,71.7916){\circle*{1.5}}
                \put(150.8468,60.5333){\circle*{1.5}}
                \put(151.0071,80.2557){\circle*{1.5}}
                \put(142.8109,82.4519){\circle*{1.5}}
                \put(145.0071,90.6480){\circle*{1.5}}
                \put(134.3468,89.1121){\circle*{1.5}}
                \put(127.6866,80.6480){\circle*{1.5}}
                \put(119.4282,68.9519){\circle*{1.5}}
                \put(159.2654,91.9519){\circle*{1.5}}

                \Line(-0.8468,60.5334)(-9.2654,91.9519)
                \Line(-0.8468,60.5334)(-1.0071,80.2558)
                \Line(-0.8468,60.5334)(5.6532,71.7917)
                \Line(-0.8468,60.5334)(16.3134,70.2558)
                \Line(-0.8468,60.5334)(30.5718,68.9519)
                \Line(-0.8468,60.5334)(0,0)
                \Line(22.3134,80.6481)(22.1532,100.3705)
                \Line(22.3134,80.6481)(30.5718,68.9519)
                \Line(22.3134,80.6481)(16.3134,70.2558)
                \Line(22.3134,80.6481)(14.1173,78.4519)
                \Line(22.3134,80.6481)(15.6532,89.1122)
                \Line(14.1173,78.4519)(15.6532,89.1122)
                \Line(14.1173,78.4519)(16.3134,70.2558)
                \Line(14.1173,78.4519)(5.6532,71.7917)
                \Line(14.1173,78.4519)(7.1891,82.4519)
                \Line(15.6532,89.1122)(7.1891,82.4519)
                \Line(15.6532,89.1122)(4.9929,90.6481)
                \Line(15.6532,89.1122)(22.1532,100.3705)
                \Line(22.1532,100.3705)(4.9929,90.6481)
                \Line(22.1532,100.3705)(-9.2654,91.9519)
                \Line(22.1532,100.3705)(75,129.9038)
                \Line(22.1532,100.3705)(30.5718,68.9519)
                \Line(4.9929,90.6481)(7.1891,82.4519)
                \Line(4.9929,90.6481)(-1.0071,80.2558)
                \Line(4.9929,90.6481)(-9.2654,91.9519)
                \Line(7.1891,82.4519)(5.6532,71.7917)
                \Line(7.1891,82.4519)(-1.0071,80.2558)
                \Line(-1.0071,80.2558)(-9.2654,91.9519)
                \Line(-1.0071,80.2558)(5.6532,71.7917)
                \Line(5.6532,71.7917)(16.3134,70.2558)
                \Line(16.3134,70.2558)(30.5718,68.9519)
                \Line(30.5718,68.9519)(75,129.9038)
                \Line(30.5718,68.9519)(0,0)
                \Line(-9.2654,91.9519)(0,0)
                \Line(-9.2654,91.9519)(75,129.9038)
                \Line(0,0)(75,129.9038)
                \put(75,129.9038){\circle*{1.5}}
                \put(0,0){\circle*{1.5}}
                \put(-0.8468,60.5334){\circle*{1.5}}
                \put(22.3134,80.6481){\circle*{1.5}}
                \put(14.1173,78.4519){\circle*{1.5}}
                \put(15.6532,89.1122){\circle*{1.5}}
                \put(22.1532,100.3705){\circle*{1.5}}
                \put(4.9929,90.6481){\circle*{1.5}}
                \put(7.1891,82.4519){\circle*{1.5}}
                \put(-1.0071,80.2558){\circle*{1.5}}
                \put(5.6532,71.7917){\circle*{1.5}}
                \put(16.3134,70.2558){\circle*{1.5}}
                \put(30.5718,68.9519){\circle*{1.5}}
                \put(-9.2654,91.9519){\circle*{1.5}}
            \end{picture}
            \caption{The graph $G_{39}$ that is a counterexample to the AB conjecture.}
            \label{figure:G_39}
        \end{figure}
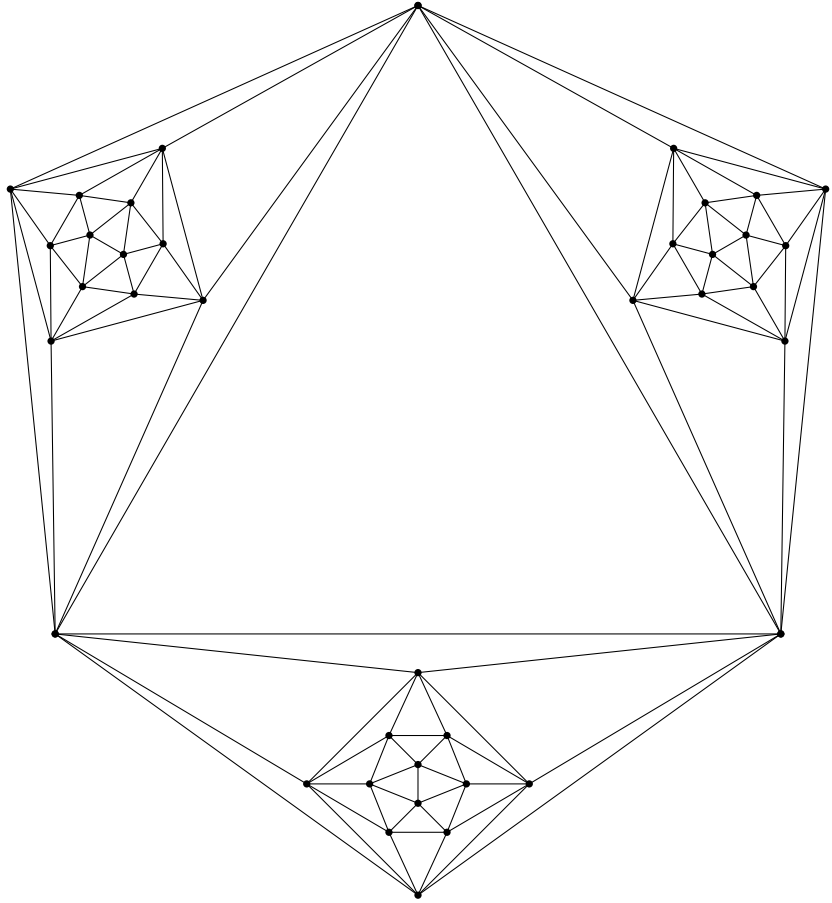

        \begin{theorem}\label{theorem:G_39_a_counterexample}
            We have $a(G_{39})=19$ and $r_a(G_{39})=\frac{19}{39}<\frac{1}{2}$. Thus, $G_{39}$ is a counterexample to the AB conjecture, and the AB conjecture does not hold.
        \end{theorem}
        \begin{proof}
            By Lemma~\ref{lemma:G_14_a}, we have $a(G_{14})=7$. By Lemma~\ref{lemma:G_14_induced_forests_7}, $G_{14}$ is all-realizable $(1,2)$-edge-restricting with the restricting vertices $w_1$ and $w_2$. Now, by Lemma~\ref{lemma:augmenting_1-2-edge-restricting_to_counterexample}, we have $a(G_{39})=3a(G_{14})-2=3 \times 7-2=19$ and hence $r_a(G_{39})=\frac{19}{39}$.
        \end{proof}

        \begin{corollary}\label{corollary:counterexamples_19/39}
            There exists an infinite sequence of planar graphs with $r_a(G)=\frac{19}{39}$.
        \end{corollary}
        \begin{proof}
            The statement follows by taking a disjoint union of an arbitrary number of copies of $G_{39}$ and using the known fact that $a(G)$ is additive for a disjoint union of graphs~\cite[Lemma~5.23, p.~108]{makarov_ab_conjecture_for_multigraphs}.
        \end{proof}

    \section{Optimization of the ratio}
        We omit the calculation and the proof here, but using a $K_4$ instead of a triangle and attaching a copy of $G_{14}$ by the restricting edge to each edge of the $K_{4}$ gives a planar graph $G_{76}$ on $76$ vertices with $a(G_{76})=37$ and $r_a(G_{76})=\frac{37}{76}$. So, $G_{76}$ has a larger number of vertices than $G_{39}$, but a slightly smaller ratio $r_a(G_{76})$.

        Another potential way of constructing planar graphs with smaller ratio $r_a(G)$ is to try to use Lemma~\ref{lemma:augmenting_1-2-edge-restricting_to_counterexample} iteratively to construct a recursive sequence of graphs with decreasing ratio $r_a(G)$ on each step. For that to work, we need to somehow make the augmented graph to also be $(1,2)$-edge-restricting. However, we were not able to obtain such a recursive construction with the ratio $r_a(G)$, even in the limit, smaller than for the non-recursive construction of $G_{76}$.

    \section{Remaining open problems}
        Even though the AB conjecture is resolved, there is still a gap between the best known lower bound $a(G) \geq \frac{2}{5}n$ and the values of $a(G)$ from the constructions.
        
        \begin{openproblem}
            Find the best possible lower bound on $a(G)$ in terms of $n$ for every planar graph on $n$ vertices.
        \end{openproblem}

        \begin{openproblem}
            Find a proof of the lower bound $\alpha(G) \geq \frac{n}{4}$ for every planar graph on $n$ vertices that does not use the Four Color Theorem.
        \end{openproblem}

    \printbibliography
\end{document}